\documentclass[11pt]{amsart}  
\usepackage{amssymb}  
\usepackage{amsthm}
\usepackage{amscd}
\usepackage{mathrsfs}
\usepackage{amsfonts}
\usepackage{amsmath}
\usepackage{graphicx}
\usepackage{pinlabel}
\usepackage{enumerate}

\DeclareMathOperator{\tr}{Tr}

\newtheorem{Thm}{Theorem}

\newtheorem{Lem}[Thm]{Lemma}
\newtheorem{Cor}[Thm]{Corollary}

\theoremstyle{definition}
\newtheorem{Def}[Thm]{Definition}

\theoremstyle{remark}

\theoremstyle{definition}

\begin{document}  

\author{Diana Hubbard}
\title[On SKH and Axis-Preserving Mutations]{On Sutured Khovanov Homology and Axis-Preserving Mutations}

\thispagestyle{empty}

\begin{abstract}This paper establishes that sutured annular Khovanov homology is not invariant for braid closures under axis-preserving mutations. This follows from an explicit relationship between sutured annular Khovanov homology and the classical Burau representation for braid closures.   
\end{abstract}

\maketitle

\section{Introduction}

Given a link $L$ embedded in $S^{3}$, locate a sphere $C \subset S^{3}$, called a Conway sphere, such that $C$ is transverse to $L$ and $|C \cap L|=4$.  

\begin{Def}\label{Definition 1} A \textbf{mutation} of $L$ is obtained as follows: cut along $C$, rotate 180 degrees about an axis disjoint from $C \cap L$ that preserves $C \cap L$ setwise, and reglue $C$ to produce a new link $L'$. 
\end{Def}

The links $L$ and $L'$ are said to be mutants.  Knot and link invariants often have difficulty distinguishing mutant knots and links.  For example, the Jones, Alexander, and HOMFLY polynomials are invariant under mutation (see \cite{kanenobu1995homfly}).  Many results are now known about the behavior of Khovanov homology under mutation, though it is still open whether Khovanov homology with $\mathbb{Z}$ coefficients distinguishes mutant knots.  Khovanov homology with $\mathbb{Z}/2\mathbb{Z}$ coefficients was shown to be invariant under mutation by Bloom (\cite{bloom2009odd}) and independently by Wehrli, and Khovanov homology with $\mathbb{Z}$ coefficients was shown by Wehrli to generally detect mutations that switch components of a link (\cite{wehrli2003khovanov}).  

Sutured annular Khovanov homology (hereafter referred to as sutured Khovanov homology for simplicity, and denoted $SKh$) is a triply graded Khovanov type invariant for knots or links embedded in a thickened annulus $A \times I$.  It is described in more detail in Section 2.  In this setting, we will be interested in a specific type of mutation that we call axis-preserving. Given a link $L \subset A \times I \subset S^{3}$, denote the point at the center of the annulus $A$ as $z$. 

\begin{Def}\label{axis-preserving} An \textbf{axis-preserving mutation} is a mutation as in Definition \ref{Definition 1} such that the axis of rotation  contains the line segment $z \times I$.
\end{Def}

\textbf{Remark:} In general, axis-preserving mutations may change the isotopy class of the link in $A \times I$ while preserving the isotopy class in $S^{3}$.  Indeed, all of the examples presented in this paper are of this type.

It is an observation of Wehrli that sutured Khovanov homology is not invariant under axis-preserving mutation for knots.  His example is shown in Figure \ref{wehrli}, where the dotted circle denotes the intersection of the Conway sphere with the annulus.  After undoing the trivial kink in both knots, we see that this mutation switches a negatively stabilized unknot with a positively stabilized unknot. A quick calculation yields that these two have distinct $SKh$.

\begin{figure}[ht!]
\labellist
\small\hair 2pt

\endlabellist
\centering
\includegraphics[scale=0.65]{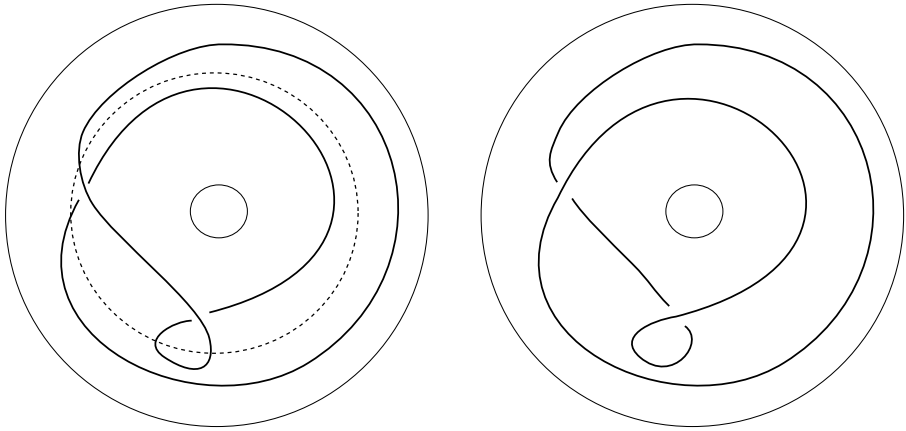}
\caption{Wehrli's example}\label{wehrli}
\end{figure}

This example shows that in the annular setting mutation is potentially a strong move on knots, since adding trivial kinks allows us to switch crossings.  Hence it is perhaps not surprising that $SKh$ can distinguish annular knots or links under axis-preserving mutation.  It is natural to ask how sutured Khovanov homology behaves under axis-preserving mutation on  braids, where such trivial kinks are not allowed. Indeed, $SKh$ is a natural tool for studying braids, as it is by construction a conjugacy class invariant.  Baldwin and Grigsby have shown that $SKh$ can detect the trivial braid among braid closures (\cite{baldwin2012categorified}), and Grigsby and Ni have shown that $SKh$ can distinguish braids from other tangles (\cite{grigsby2013sutured}).  

Throughout this paper we assume that braids are embedded in the natural way in $A \times I$: that is, the axis from Definition \ref{axis-preserving} is precisely the braid axis.  Exchange moves and flypes on closed braids are special cases of axis-preserving mutations (see Figure \ref{exchange_flype}: the x  represents the braid axis, and the $w$ stands for $w$ strands).

\begin{figure}[ht!]
\labellist
\small\hair 1pt
\pinlabel $Q$ at 48 94
\pinlabel $w$ at 9 90
\pinlabel $P$ at 48 8
\pinlabel $Q$ at 185 94
\pinlabel $w$ at 145 90
\pinlabel $P$ at 185 9
\pinlabel $Q$ at 279 75
\pinlabel $P$ at 279 27
\pinlabel $R$ at 347 52
\pinlabel $Q$ at 417 75
\pinlabel $P$ at 417 27
\pinlabel \rotatebox[origin=c]{180}{$R$} at 433 53

\endlabellist
\centering
\includegraphics[scale=0.72]{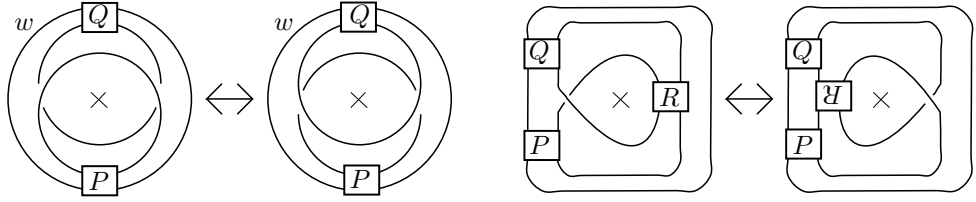}
\caption{Exchange moves (left) and negative flypes (right)}\label{exchange_flype}
\end{figure}


The main result of this paper is:

\begin{Thm}\label{mutation}
The sutured Khovanov homology of a closed braid is not invariant under an axis-preserving mutation.  Indeed, there exist infinite families of mutant 4-braid pairs and mutant 5-braid pairs, shown in Figures \ref{morton} and \ref{menasco}, whose sutured Khovanov homologies differ.
\end{Thm}

\begin{figure}[ht!]
\labellist
\small\hair 2pt
\pinlabel $\sigma_{1}^{k}$ at 175 314

\endlabellist
\centering
\includegraphics[scale=0.65]{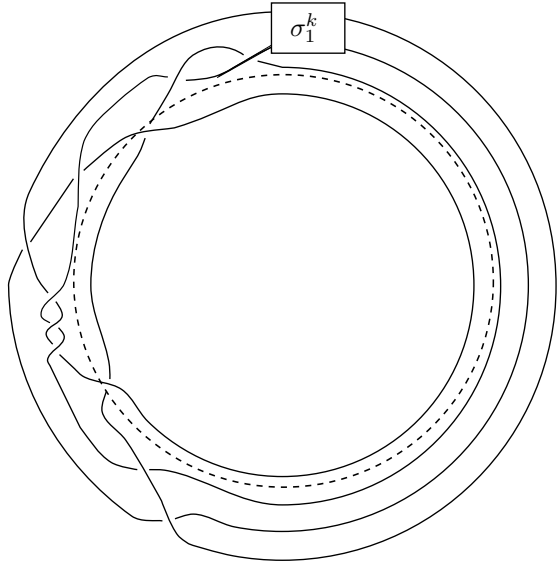}
\caption{Infinite family of 4-braid mutants whose $SKh$ differ; $k$ is an integer $\geq 0$}\label{morton}
\end{figure}

\begin{figure}[ht!]
\labellist
\small\hair 2pt
\pinlabel $\sigma_{1}^{k}$ at 186 270
\pinlabel $\sigma_{1}^{-k}$ at 188 16
\endlabellist
\centering
\includegraphics[scale=0.65]{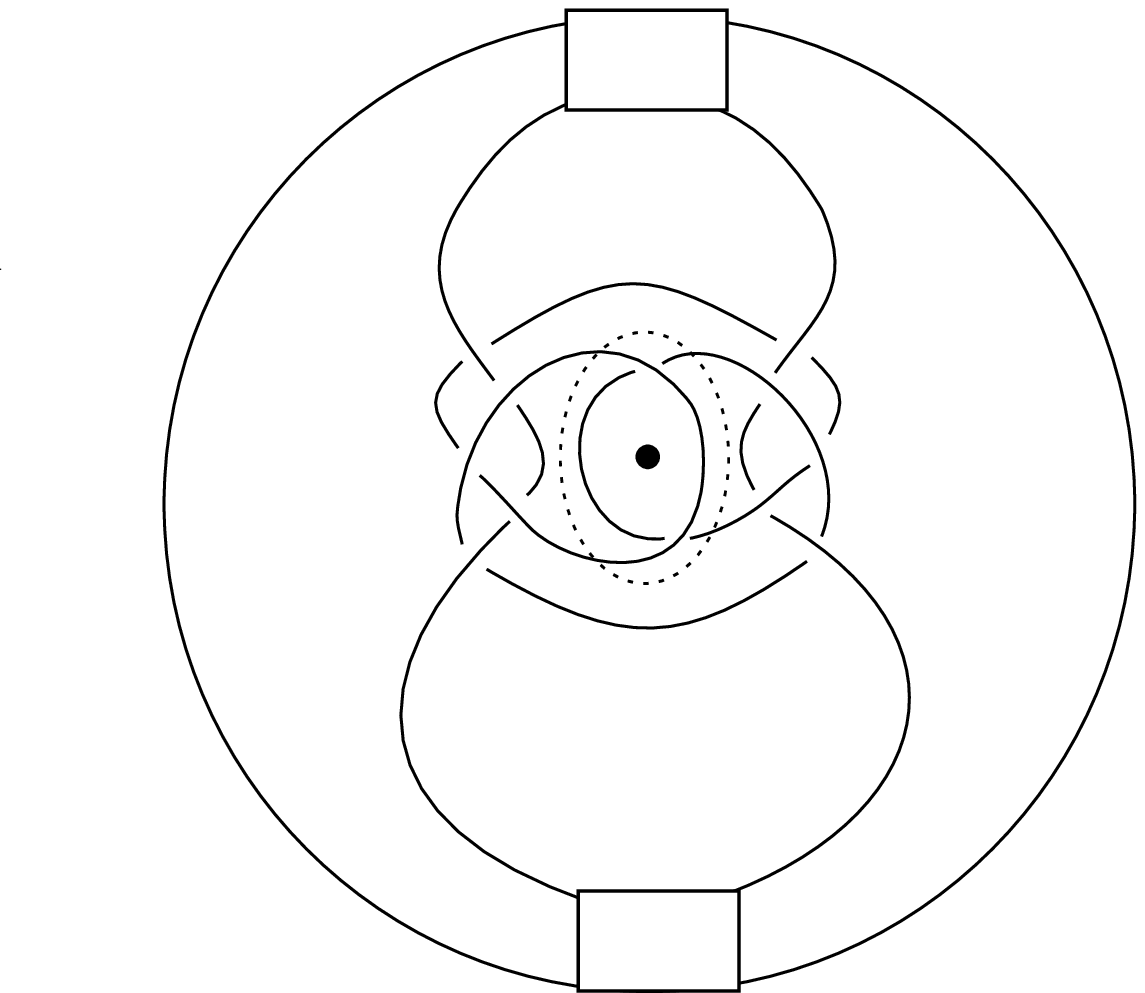}
\caption{Infinite family of 5-braid mutants whose $SKh$ differ; $k$ is an integer $\geq 0$}\label{menasco}
\end{figure}

Again, the dotted circles in Figures \ref{morton} and \ref{menasco} represent the intersection of the Conway sphere with the annulus.

\textbf{Note:} The mutant pairs in Theorem \ref{mutation} are related by exchange moves.  Sutured Khovanov homology can also distinguish mutants that are related by negative flypes. In particular, the pair
$$ \sigma_{3}^{2}\sigma_{2}^{2}\sigma_{3}^{-1}\sigma_{1}^{2}\sigma_{2}\sigma_{1}^{-1} \,\, \text{and} \,\, \sigma_{3}^{2}\sigma_{2}^{2}\sigma_{3}^{-1}\sigma_{1}^{-1}\sigma_{2}\sigma_{1}^{2}$$
from  \cite{ng2008transverse}, both representing the knot $7_{2}$ and related by a negative flype, are distinguished by $SKh$.  In addition, the pair
 $$\sigma_{3}\sigma_{2}^{-2}\sigma_{3}^{2}\sigma_{2}\sigma_{3}^{-1}\sigma_{1}^{-1}\sigma_{2}\sigma_{1}^{2} \,\, \text{and} \,\, \sigma_{3}\sigma_{2}^{-2}\sigma_{3}^{2}\sigma_{2}\sigma_{3}^{-1}\sigma_{1}^{2}\sigma_{2}\sigma_{1}^{-1}$$ from \cite{khandhawit2010family} (see also \cite{ng2008transverse}), both representing $10_{132}$ and related by a negative flype, are also distinguished by $SKh$. 
 
It is worth emphasizing here that $SKh$ is \textit{not} a transverse invariant (since in general it is not preserved under positive stabilization).  So in particular, we cannot conclude from this calculation that the pairs mentioned above are transversely non-isotopic. 

The example shown in Figure \ref{morton} appears in \cite{birman1992studying} with $k=0$ (building on work of Morton in \cite{morton1983irreducible}) as the intermediate and third braid in a series of three braids related by exchange moves, where the first does not admit a destabilization and the third does.  The example shown in Figure \ref{menasco} was suggested by Menasco.

Given a braid $\beta$ and its associated closed braid $\overline{\beta}$, denote the classical unreduced Burau representation (as described in \cite{birman2005braids}) with variable $t$ as $\Phi(\beta,t)$.  We denote the graded Euler characteristic of $SKh(\overline\beta)$ as $\chi_{SKh}(\overline{\beta})$.  The relationship between the Burau representation of a braid and the $U_q(sl_2)$ Reshetikhin-Turaev invariant is well-known among experts (cf. \cite{jones1987hecke}, \cite{jackson2001braid}, \cite{jackson2011lawrence}), and a relationship between the  Reshetikhin-Turaev invariant and the graded Euler characteristic of sutured Khovanov homology is described in \cite{grigsby2011gradings}, building on work in Khovanov's thesis \cite{khovanov1997graphical} (see also \cite{khovanov2002quivers}).  In this paper we recover the relationship between the Burau representation and sutured Khovanov homology explicitly in Khovanov's diagrammatic language:

\begin{Thm}\label{trace} Given an $n$-braid $\beta$ with $n_{+}$ positive crossings and $n_{-}$ negative crossings,
$$ \chi_{SKh}(\overline{\beta})|_{k=n-2} = (qt)^{n-2}(q)^{n_{+}-n_{-}} \tr \left(\Phi(\beta,q^{2})\right) $$
\end{Thm}
That is, the trace of the Burau representation of a braid can be recovered from the $SKh$ of its closure.

  Theorem \ref{mutation} is a consequence of Theorem \ref{trace} along with calculations of the corresponding traces. Indeed, Theorem \ref{trace} gives a useful method for distinguishing the sutured Khovanov homologies of some braids:

\begin{Cor}\label{exponent} Suppose two $n$-braids $\beta_{1}$ and $\beta_{2}$ have the same exponent sum.  If the traces of the Burau representations of $\beta_{1}$ and $\beta_{2}$ differ, then the sutured Khovanov homologies of the two braids differ as well.
\end{Cor}

Also note that $SKh$ cannot \textit{always} distinguish mutant braids: in Corollary 2 of \cite{baldwin2012categorified}, Baldwin and Grigsby (using a result of Birman and Menasco in \cite{birman2008note}) proved that there exist infinitely many pairs of non-conjugate mutant 3-braids with the same $SKh$.

\vspace{11pt}
\textbf{Acknowledgments:} The author thanks Eli Grigsby for her guidance and support throughout this project.  The author also thanks John Baldwin and Eli Grigsby for suggesting the topic, John Baldwin and Hoel Queffelec for many helpful conversations, Andrew Phillips for answering several of her questions, and Stephan Wehrli for useful correspondence.  This work was partially supported by NSF CAREER award DMS-1151671.
\section{Sutured Khovanov Homology}

We assume throughout that homology is computed with coefficients in $\mathbb{Z}/2\mathbb{Z}$.  

Sutured Khovanov homology, first constructed in \cite{asaeda2004categorification} and related to knot Floer homology by Roberts in \cite{roberts2007knot}, is an invariant for links in a thickened annulus $A \times I$.  
Specifically, $A \times I$ is embedded in $\mathbb{R}^{2} \times \mathbb{R}$, and we project a link $\mathbb{L}$ into $A$ embedded in $\mathbb{R}^{2}$.  As in the construction of Khovanov homology, we pick an order for the $n$ crossings of a projection of $\mathbb{L}$, and associate to each vertex of the cube $\{0,1\}^{n}$ a resolution of $\mathbb{L}$ by resolving the crossings of the projection according to the rule in Figure \ref{resolutions}.  The resulting object is called the cube of resolutions of $\mathbb{L}$.

\begin{figure}[h!]
\labellist
\small\hair 2pt
\pinlabel $0$ at 62 80
\pinlabel $1$ at 62 28
\endlabellist
\centering
\includegraphics[scale=0.7, clip=false]{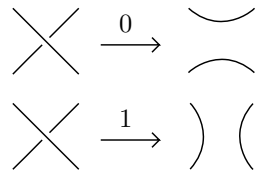}
\caption{Resolutions of crossings}\label{resolutions}
\end{figure}

One now associates a chain complex to the cube of resolutions. A circle in a resolution is said to be trivial if it bounds a disk in $A$, and non-trivial if not.  As in Khovanov homology, to each circle in a resolution we assign a copy of $V$, a vector space generated by the two basis elements $v_{+}$ and $v_{-}$.  Here $V$ is endowed with two gradings, the standard Khovanov $q$-grading and an extra $k$-grading. The $k$-grading is assigned as follows: $gr_{k}(v_{\pm}) =  \pm 1$ when the corresponding $V$ is assigned to a non-trivial circle, and $gr_{k}(v_{\pm}) =  0$ when the corresponding $V$ is assigned to a trivial circle.

 To each vertex of the cube we associate the vector space $V^{\otimes c}\{i(\mathcal{I})\}$, where $c$ is the number of circles in that resolution, and $\{i(\mathcal{I})\}$ denotes a shift in the $q$-grading by the height $i(\mathcal{I})$, that is, $i(\mathcal{I})$ is the number of $1$'s in the vertex $\mathcal{I}$.  

The standard Khovanov differential is non-increasing in the $k$-grading, which induces a filtration on Khovanov's chain complex.  After an overall shift of $\{n_{+} - 2n_{-}\}$ in the $q$-grading, the homology of the associated graded chain complex is $SKh(\mathbb{L})$.  This is a triply graded invariant of $\mathbb{L}$ in $A \times I$.

The graded Euler characteristic of $SKh(\mathbb{L})$ is:
$$ \chi_{SKh}(\mathbb{L}) = \sum_{i,j,k} (-1)^{i}q^{j}t^{k} \text{dim}(SKh(\mathbb{L}^{i,j,k}))$$
where $SKh(\mathbb{L})^{i,j,k}$ is the homogeneous component of $SKh(\mathbb{L})$ in homological grading $i$, $q$-grading $j$, and $k$-grading $k$.

We now describe the construction of the Reshetikhin-Turaev invariant, a $U_{q}(sl_2)$-module map that is an invariant of tangles.  Given a tangle $\mathbb{T}$, it is calculated by constructing a matrix $J(\mathbb{T})$ and multiplying by final shifts of $(-1)^{n_{-}}(q)^{n_{+}-2n_{-}}$ to ensure invariance under the Reidemeister moves.  We will refer to the matrix $J(\mathbb{T})$ as the Reshetikhin-Turaev matrix of $\mathbb{T}$.  In what follows, we restrict our attention to braids and their closures and take advantage of some simplifications in notation that this yields.  Refer to \cite{grigsby2011gradings} for a more general description.  

Recall from Grigsby and Wehrli \cite{grigsby2011gradings}, building on work in \cite{khovanov1997graphical},  that the graded Euler characteristic of the sutured Khovanov homology of a closed braid $\overline\beta$ can be calculated using the Reshetikhin-Turaev matrix $J(\beta)$ of any associated braid $\beta$ whose closure is isotopic to $\overline\beta$ in $A \times I$  ($n_{+}$ denotes the number of positive crossings and $n_{-}$ the number of negative crossings):

\begin{align*}
\chi_{SKh}(\overline\beta) 
& = (-1)^{n_{-}}(q)^{n_{+}-2n_{-}}\sum_{k} (qt)^{k} \tr(J(\beta)|_{[k]})
\end{align*}

For an $n$-braid $\beta$, the Reshetikhin-Turaev matrix $J(\beta)$ is a $U_{q}(sl_{2})$-module map $V_{1}^{\otimes n} \to V_{1}^{\otimes n}$ intertwining the quantum group action.  Here $V_{1}$ is the two-dimensional fundamental representation of $U_{q}(sl_{2})$ with underlying vector space $V_{1} = \mathbb{C}(q)v_{+} \oplus \mathbb{C}(q)v_{-}$.  The generators $E, F, K$ of $U_{q}(sl_2)$ act by
$$Kv_{+} = qv_{+}, \,\,\,\,\, Kv_{-} = q^{-1}v_{-}, \,\,\,\,\, Ev_{+} = Fv_{-} = 0, \,\,\,\,\, Ev_{-} = v_{+}, \,\,\,\,\, Fv_{+} = v_{-}. $$

In general, $J(\beta)$ is constructed diagrammatically by first calculating a matrix $J(\mathbb{\beta}^{I})$ for each choice of resolution $I$ of $\beta$; the matrices associated to the resolutions are combined via $$J(\mathbb{\beta}) = \sum_{I} (-q)^{i(I)} J(\mathbb{\beta}^{I})$$ 
	Each $J(\mathbb{\beta}^{I})$ is calculated as follows.  We choose a basis for $V_{1}^{\otimes n}$ in one-to-one correspondence with $n$-tuples in $\{ \uparrow, \downarrow \}^{n}$, namely by identifying $\uparrow$ with $v_{+}$ and identifying $\downarrow$ with $v_{-}$.  For every $\textbf{i}, \textbf{j} \in \{\uparrow, \downarrow\}^{n}$, orient the top of $\beta^{I}$ locally with $\bf{i}$ and the bottom of $\beta^{I}$ locally with $\bf{j}$, reading left to right. The ($\textbf{i}, \textbf{j}$) entry of $J(\beta^{I})$ is zero if any of the orientations in the resulting  diagram are incompatible with each other.  

If the orientations are compatible, we form the set $E_{\textbf{i,j}}(\beta^{I}) = \{E(\beta^{I}): t(\beta^{I}) = $ $\textbf{i}$, $b(\beta^{I}) = $ $\textbf{j}\}$ of all possible orientations of $\beta^{I}$ satisfying that the top orientation of $\beta^{I}$ is $\bf{i}$ and the bottom orientation of $\beta^{I}$ is $\bf{j}$.  (For example, if $\beta^{I}$ contains precisely one closed component, $E(\beta^{I})$ contains two elements, one for each orientation of the closed component). Then the $\bf{i}, \bf{j}$ entry of $J(\beta^{I})$ is a weighted sum over all elements of $E_{\textbf{i,j}}(\beta^{I})$:
$$ J(\beta^{I})_{\textbf{i,j}} = \sum_{\mathbb{S} \in E_{\textbf{i,j}}(\beta^{I})} q^{j(\mathbb{S})}$$
where we describe $q^j(\mathbb{S})$, the appropriate power of $q$, here.

For each element in $E_{\textbf{i,j}}(\beta^{I})$, every arc is assigned a $q^{0}$ unless we have one of the cases shown in Figure \ref{rt}, in which case the assignment is as shown.
\begin{figure}[ht!]
\labellist
\small\hair 2pt
\pinlabel $q^{-1}$ at 55 34
\pinlabel $q$ at 206 31
\endlabellist
\centering
\includegraphics[scale=0.65]{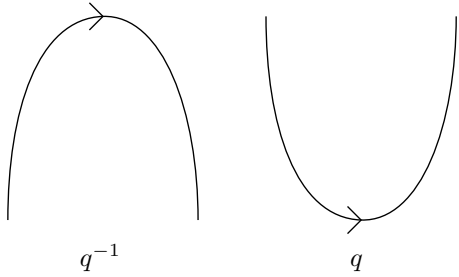}
\caption{Assignments in the two special cases}\label{rt}
\end{figure}
We multiply the assignments corresponding to each arc in the diagram to obtain a single power of $q$, written $q^{j(\mathbb{S})}$ for every element $\mathbb{S}$ in $E_{\textbf{i,j}}(\beta^{I})$.  

The $k$ in $\chi_{SKh}$ corresponds to $k = \# (\uparrow) - \# (\downarrow)$ in $\{\uparrow, \downarrow\}^{n}$.  Notice that the $(\bf{i},\bf{j})$ entry in $J(\mathbb{\beta})$ is zero if $k(\bf{i})$ $\neq k(\bf{j})$.  This implies that $J(\mathbb{\beta})$ is a block diagonal matrix, with a block for each $k$. We denote the block sub-matrix corresponding to a fixed $k$ by $J(\beta)|_{[k]}$. 

\section{Results}

In this section we prove the main results.  We start with Theorem \ref{trace}, which we restate here for reference:

\vspace{11 pt}
\noindent
\textbf{Theorem \ref{trace}.} \textit{Given a braid $\beta$, $$\chi_{SKh}(\overline{\beta})|_{k=n-2} = (qt)^{n-2}(q)^{n_{+}-n_{-}} \tr \left(\Phi(\beta,q^{2})\right) $$}

In order to prove Theorem \ref{trace}, we first observe that the Reshetikhin-Turaev matrix takes a particularly nice form for braids when we restrict to $k = n-2$. 

\begin{Lem}\label{matrix} Consider the standard Artin generators $\sigma_{1}, \ldots, \sigma_{n-1}$ for the $n$-strand braid group.  Then $J(\sigma_{i})|_{[n-2]}$ is an $n$ by $n$ matrix that takes the following form: there is a 2 by 2 block
$$ \left(\begin{matrix} 1 &0 \\0 & 1 \end{matrix} \right) -q\left(\begin{matrix} q & 1  \\ 1 & q^{-1} \end{matrix} \right) =  q^{-1} \left(\begin{matrix} q-q^{3} & -q^{2}  \\ -q^{2} & 0  \end{matrix} \right)$$
with the $0$ in the $(i,i)$ spot and $q^{-1}(q)$ along the diagonal everywhere else.

$J(\sigma_{i}^{-1})|_{[n-2]}$ takes the following form: there is a 2 by 2 block
$$ \left (\begin{matrix} q & 1  \\ 1 & q^{-1} \end{matrix} \right) -q\left(\begin{matrix} 1 &0 \\0 & 1 \end{matrix} \right) = -q^{2} \left( \begin{matrix} 0 & -q^{-2} \\ -q^{-2} & -q^{-3}+q^{-1} \end{matrix} \right)$$
with the $0$ in the $(i,i)$ spot, and $-q^{2}(q^{-1})$ along the diagonal everywhere else.  

For ease of notation in later calculations, we name the factored matrices $L$: e.g., $J(\sigma_{i})|_{[n-2]} = q^{-1}L(\sigma_{i}))$ and $J(\sigma_{i}^{-1})|_{[n-2]} = -q^{2}L(\sigma_{i}^{-1})$.
\end{Lem}

For example, for $\sigma_{2} \in B_{4}$, 
$$ J(\sigma_{2})|_{[2]} = q^{-1} \left(\begin{matrix} q & 0 & 0 & 0 \\ 0 & q-q^{3} & -q^{2} & 0 \\ 0 & -q^{2} & 0 & 0 \\ 0 & 0 & 0 & q \end{matrix}\right) = q^{-1}L(\sigma_{2})$$

We note again here that the Reshetikhin-Turaev \textit{matrix} is not an invariant of braids. For example, $J(\sigma_{i})|_{[n-2]}J(\sigma_{i}^{-1})|_{[n-2]}$ is not the identity.  However, the final grading shift $(-1)^{n_{-}}(q)^{n_{+}-2n_{-}}$ removes both the $q^{-1}$ coefficient from each positive generator and the $-q^{2}$ coefficient from each negative generator, taking care of this problem:
$$\resizebox{1\hsize}{!}{$\displaystyle{(-1)^{1}(q)^{1-2}J(\sigma_{i})|_{[n-2]}J(\sigma_{i}^{-1})|_{[n-2]} = (-1)^{1}(q)^{1-2}(-1)^{1}(q)^{2-1} L(\sigma_{i})L(\sigma_{i}^{-1}) =  I}$}$$

\begin{proof}[Proof of Lemma \ref{matrix}]

This is a calculation.  We show it here for $\sigma_{1}$ in $B_{3}$; it will be clear how to extend the calculation to more strands and other $\sigma_{i}$.  Restricting to $k=1$, we choose the basis ordering $\{\downarrow \uparrow \uparrow, \uparrow \downarrow \uparrow, \uparrow \uparrow \downarrow\}$. For example, for $\sigma_{1} \in B_{3}$, the $(1,2)$ entry of $J(\sigma_{1})|_{[1]}$  corresponds to orienting the top strands with $\downarrow \uparrow \uparrow$ and the bottom strands with $\uparrow \downarrow \uparrow$. 

Calculating using the rules described in Section 2, the matrix associated to the $0$-resolution of $\sigma_{1}$ is the identity matrix, and the matrix associated to the $1$-resolution of $\sigma_{1}$ is 
$$\left(\begin{matrix} q & 1 & 0 \\ 1 & q^{-1} & 0 \\ 0 & 0 & 0 \end{matrix} \right)$$
since the resolutions associated to $\sigma_{1}$ are as shown in Figure \ref{sigma1}.

\begin{figure}[ht!]
\labellist
\small\hair 2pt
\pinlabel $\sigma_{1}$ at 40 25
\pinlabel 0-resolution at 161 25
\pinlabel 1-resolution at 280 25
\endlabellist
\centering
\includegraphics[scale=0.65]{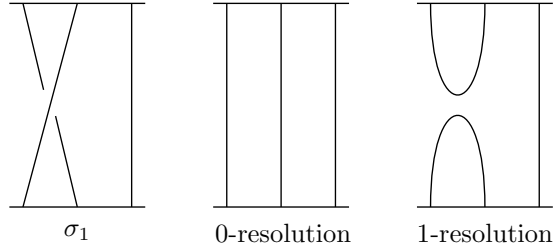}
\caption{The resolutions of $\sigma_{1} \in B_{3}$}\label{sigma1}
\end{figure}

So $J(\sigma_{1})|_{[1]}$ is:
$$  J(\sigma_{1}^{0})_{[1]} - qJ(\sigma_{1}^{1})_{[1]} = \left(\begin{matrix} 1 & 0 & 0 \\ 0 & 1 & 0 \\ 0 & 0 & 1 \end{matrix} \right) - q\left(\begin{matrix} q & 1 & 0 \\ 1 & q^{-1} & 0 \\ 0 & 0 & 0 \end{matrix}\right)
= q^{-1}\left(\begin{matrix} q-q^{3} & -q^{2} & 0 \\ -q^{2} & 0 & 0 \\ 0 & 0 & q \end{matrix}\right)$$
For $\sigma_{1}^{-1}$, the only difference is that the $0$ and $1$ resolutions are switched.

\end{proof}

We now show that given a braid $\beta$ with braid word in the standard Artin generators, we can find $J(\beta)|_{[n-2]}$ by composing the matrices for each braid generator (as described in Lemma \ref{matrix}).

\begin{Lem}  Given a braid $\beta = \sigma_{l_{1}}\sigma_{l_{2}} \cdots \sigma_{l_{m}}$,
$$ J(\beta)|_{[n-2]} = J(\sigma_{l_{1}})|_{[n-2]}J(\sigma_{l_{2}})|_{[n-2]}\cdots J (\sigma_{l_{m}})|_{[n-2]}$$
\end{Lem}

\begin{proof}
Recall that $J(\beta)$ is a block diagonal matrix with a block for each $k$; hence we can restrict to a specific $k$ for matrix operations. In what follows we drop the $k=n-2$ notation for simplicity.  We prove this by induction on the length of the braid word.  The base case is trivial. 

Any given resolution $\beta^{\mathcal{I}}$ of $\beta$ restricts to a resolution of each individual braid generator; we write $\beta^{\mathcal{I}} = \sigma_{l_{1}}^{\mathcal{I}}\cdots\sigma_{l_{m}}^{\mathcal{I}}$ where by $\sigma_{l_{j}}^{\mathcal{I}}$ we mean $\sigma_{l_{j}}$ restricted to the $j$'th entry of $\mathcal{I}$ (either a 0 or a 1).  We first show the result on the resolution level: it suffices to show that 
$$ J(\beta^{\mathcal{I}}) = J(\sigma_{l_{1}}^{\mathcal{I}}\cdots\sigma_{l_{m-1}}^{\mathcal{I}}) J(\sigma_{l_{m}}^{\mathcal{I}})$$  We have $$
J(\beta^{\mathcal{I}})_{\textbf{i,j}} = \sum_{\mathbb{S} \in E_{\textbf{i,j}}(\beta^{\mathcal{I}})} q^{j(\mathbb{S})} $$
$$= \sum_{\mathbb{S''} \in E_{\textbf{k,j}}(\sigma_{l_{m}}^{\mathcal{I}})} \sum_{\mathbb{S'} \in E_{\textbf{i,k}}(\sigma_{l_{1}}^{\mathcal{I}} \cdots \sigma_{l_{m-1}}^{\mathcal{I}})} q^{j(\mathbb{S'}\mathbb{S''})}$$
where by $\mathbb{S'}\mathbb{S''}$ we mean the vertical stacking of these two enhanced resolutions.  The expression becomes:
$$\sum_{\mathbb{S''} \in E_{\textbf{k,j}}(\sigma_{l_{m}}^{\mathcal{I}})} \sum_{\mathbb{S'} \in E_{\textbf{i,k}}(\sigma_{l_{1}}^{\mathcal{I}} \cdots \sigma_{l_{m-1}}^{\mathcal{I}})} q^{j(\mathbb{S'})}q^{j(\mathbb{S''})}$$
$$= \sum_{\mathbb{S''} \in E_{\textbf{k,j}}(\sigma_{l_{m}}^{\mathcal{I}})}  q^{j(\mathbb{S''})} \sum_{\mathbb{S'} \in E_{\textbf{i,k}}(\sigma_{l_{1}}^{\mathcal{I}} \cdots \sigma_{l_{m-1}}^{\mathcal{I}})} q^{j(\mathbb{S'})} $$
$$ = \sum_{k} J(\sigma_{l_{m}}^{\mathcal{I}})_{kj} J(\sigma_{l_{1}}^{\mathcal{I}}\cdots\sigma_{l_{m-1}}^{\mathcal{I}})_{ik} = \sum_{k} J(\sigma_{l_{1}}^{\mathcal{I}}\cdots\sigma_{l_{m-1}}^{\mathcal{I}})_{ik}J(\sigma_{l_{m}}^{\mathcal{I}})_{kj}$$
This shows 
$$ J(\beta^{\mathcal{I}}) = J(\sigma_{l_{1}}^{\mathcal{I}})\cdots J(\sigma_{l_{m}}^{\mathcal{I}})$$

Now recall that 
$$ J(\beta) = \sum_{I} (-q)^{i(\mathcal{I})} J(\beta^{\mathcal{I}})$$
and
$$ J(\sigma_{l_{1}}) \cdots J(\sigma_{l_{m}}) = (J(\sigma_{l_{1}}^{0}) - qJ(\sigma_{l_{1}}^{1}))\cdots(J(\sigma_{l_{m}}^{0})-qJ(\sigma_{l_{m}}^{1}))$$
Multiplying out the second expression gives the first.

\end{proof}

Each of the factored matrices in Lemma \ref{matrix} is (up to a constant) conjugate to the Burau representation: 

\begin{proof}[Proof of Theorem \ref{trace}] We show here how $L(\sigma_{1})$ is conjugate (up to a constant power of $q$) to the Burau representation for $\sigma_{1} \in B_{3}$, and then we will see that this is true for any $\sigma_{i}^{\pm 1} \in B_{n}$ for any $n$:
$$ L(\sigma_{1}) = \left( \begin{matrix} q-q^{3} & -q^{2} & 0 \\ -q^{2} & 0 & 0 \\ 0 & 0 & q \end{matrix} \right) = q \left(\begin{matrix} 1-q^{2} & -q & 0 \\ -q & 0 & 0 \\ 0 & 0 & 1 \end{matrix} \right)$$
Set $$ A = \left(\begin{matrix} q^{-1} & 0 & 0 \\ 0 & -q^{-2} & 0 \\ 0 & 0 & q^{-3} \end{matrix}\right)$$
Then
$$ A\left(\begin{matrix} 1-q^{2} & -q & 0 \\ -q & 0 & 0 \\ 0 & 0 & 1 \end{matrix} \right)A^{-1} = \left(\begin{matrix} 1-q^{2} & q^{2} & 0 \\ 1 & 0 & 0 \\ 0 & 0 & 1 \end{matrix}\right)$$
which is the classical Burau matrix for $\sigma_{1}$ with $t=q^{2}$.  

The process is similar for $\sigma_{1}^{-1}$:
$$ L(\sigma_{1}^{-1}) = \left(\begin{matrix} 0 & -q^{-2} & 0 \\ -q^{-2} & -q^{-3}+q^{-1} & 0 \\ 0 & 0 & q^{-1} \end{matrix}\right) = q^{-1}\left(\begin{matrix} 0 & -q^{-1} & 0 \\ -q^{-1} & -q^{-2}+1 & 0 \\ 0 & 0 & 1\end{matrix}\right)$$
and
$$A\left(\begin{matrix} 0 & -q^{-1} & 0 \\ -q^{-1} & -q^{-2}+1 & 0 \\ 0 & 0 & 1\end{matrix}\right)A^{-1} = \left(\begin{matrix} 0 & 1 & 0 \\ q^{-2} & -q^{-2}+1& 0 \\ 0 & 0 & 1 \end{matrix}\right)$$

 One can check that the same conjugating matrix $A$ works for each braid generator.  Expand $A$ to work for an arbitrary number of strands by continuing alternating signs along the diagonal and ending at $q^{-n}$. 


Now we have for $\beta = \sigma_{l_{1}}\cdots\sigma_{l_{m}}$, together with the fact that trace is invariant under conjugation,
\Small{
\begin{align*}
\chi_{SKh}(\overline\beta)|_{k=n-2} & = (-1)^{n_{-}}(q)^{n_{+}-2n_{-}} (qt)^{n-2} \tr(J(\mathbb{\beta})|_{[n-2]})\\
&= (-1)^{n_{-}}(q)^{n_{+}-2n_{-}} (qt)^{n-2} \tr(J(\sigma_{l_{1}})|_{[n-2]}\cdots J (\sigma_{l_{m}})|_{[n-2]})\\
& = (-1)^{n_{-}}(q)^{n_{+}-2n_{-}} (qt)^{n-2} (-1)^{n_{-}}(q)^{2n_{-}-n_{+}} \tr(L(\sigma_{l_{1}})\cdots L(\sigma_{l_{m}}))\\
&= (-1)^{n_{-}}(q)^{n_{+}-2n_{-}} (qt)^{n-2} (-1)^{n_{-}}(q)^{2n_{-}-n_{+}} \tr(AL(\sigma_{l_{1}})A^{-1}\cdots AL(\sigma_{l_{m}})A^{-1})\\
&= (qt)^{n-2} q^{n_{+}-n_{-}} \tr(\Phi(\beta,q^{2}))
\end{align*}
}
\normalsize
proving Theorem \ref{trace}.

\end{proof} 

We are now ready to prove Theorem \ref{mutation} using Theorem \ref{trace} (in particular using Corollary \ref{exponent}, since in each example the exponent sum of the braids is preserved under the mutation in question).

\begin{proof}[Proof of Theorem \ref{mutation}] We have two families of examples.

\noindent\textbf{Example 1}: The following is an example of a family of 4-braid pairs related by a braid-axis preserving mutation whose sutured Khovanov homologies differ.  See Figure \ref{morton}.

Given $A = \sigma_{2}^{-2}\sigma_{3}\sigma_{2}^{-1}\sigma_{1}^{-1}\sigma_{2}^{3}\sigma_{3}^{-1}\sigma_{2}\sigma_{1}$ and $B = \sigma_{2}^{-2}\sigma_{3}^{-1}\sigma_{2}^{-1}\sigma_{1}^{-1}\sigma_{2}^{3}\sigma_{3}\sigma_{2}\sigma_{1}$, then the braids $(\sigma_{1})^{k}A$ and $(\sigma_{1})^{k}B$ are related by a mutation and their $SKh$ differs for all $k \geq 0$. 

For $k=0$, a calculation shows that 
$$ \tr(\Phi(A,t)) = -t^{-3}+2t^{-2}-4t^{-1}+6-5t+3t^{2}-2t^{3}+t^{4} \,\,\,\, \text{and}$$
$$ \tr(\Phi(B,t)) = -2t^{-1} + 4 - 3t + t^{2}$$
and the result follows by Corollary \ref{exponent}.

For $k \geq 1$: first, we observe that the Burau matrix for $\sigma_{1}^{k}$ is as follows:

$$\left(\begin{matrix} \sum\limits_{m=0}^{k} (-t)^{m} & \sum\limits_{m=1}^{k} (-1)^{m+1}t^{m} & 0 & 0 \\
\sum\limits_{m=0}^{k-1} (-t)^{m} & \sum\limits_{m=1}^{k-1} (-1)^{m+1}t^{m} & 0 & 0 \\
0 & 0 & 1 & 0 \\
0 & 0 & 0 & 1 \end{matrix}\right)$$

We prove this by induction on $k$. The base case $k=1$ is trivially true.  


Call the following matrix $M$:

 $$\left(\begin{matrix} \sum\limits_{m=0}^{k} (-t)^{m} & \sum\limits_{m=1}^{k} (-1)^{m+1}t^{m} & 0 & 0 \\
\sum\limits_{m=0}^{k-1} (-t)^{m} & \sum\limits_{m=1}^{k-1} (-1)^{m+1}t^{m} & 0 & 0 \\
0 & 0 & 1 & 0 \\
0 & 0 & 0 & 1 \end{matrix}\right)\left(\begin{matrix} 1-t & t & 0 & 0 \\ 1 & 0 & 0 & 0 \\ 0 & 0 & 1 & 0 \\ 0 & 0 & 0 & 1 \end{matrix} \right)$$

We prove the result by examining the entries of $M$.  We show one entry here; the rest are similar.

$$M_{(1,1)} = (1-t)\sum\limits_{m=0}^{k} (-t)^{m} + \sum\limits_{m=1}^{k} (-1)^{m+1}t^{m} =$$ $$ \sum\limits_{m=0}^{k} (-t)^{m} + \sum\limits_{m=0}^{k} (-1)^{m+1}t^{m+1} + \sum\limits_{m=1}^{k} (-1)^{m+1}t^{m} = $$
$$ = \sum\limits_{m=1}^{k} (-t)^{m} + \sum\limits_{m=1}^{k+1} (-1)^{m}t^{m} + \sum\limits_{m=1}^{k} (-1)^{m+1}t^{m} =$$ $$ \sum_{m=0}^{k} (-t)^{m} + (-1)^{k+1}t^{k+1} + \sum\limits_{m=1}^{k} t^{m}((-1)^{m} + (-1)^{m+1})$$
Since $(-1)^{m} + (-1)^{m+1} = 0$, the desired result follows.

Now that we have established a formula for the Burau matrix of $\sigma_{1}^{k}$, we examine $\tr(\Phi(\sigma_{1}^{k}A),t)$ and $\tr(\Phi(\sigma_{1}^{k}B),t)$ (Mathematica gives us the entries in the matrices for $A$ and $B$):
$$\tr(\Phi(\sigma_{1}^{k}A),t) = \left(\sum\limits_{m=0}^{k}(-t)^{m}\right)(1-2t+2t^{2}-2t^{3}+t^{4}) + $$ $$\left(\sum\limits_{m=1}^{k}(-1)^{m+1}t^{m}\right)(-t^{-2}+3t^{-1}-4+3t-t^{2}) +$$ $$ \left(\sum\limits_{m=1}^{k} (-1)^{m+1}t^{m} \right)(-t^{-2}+2t^{-1}-1) -t^{-3}+3t^{-2}-6t^{-1}+6-3t+t^{2}$$
and
$$ \tr(\Phi(\sigma_{1}^{k}B),t) = \left(\sum\limits_{m=0}^{k} (-t)^{m}\right)(1-t) + 3-2t^{-1}-2t+t^{2}$$
The largest power of $t$ appearing in $\tr(\Phi(\sigma_{1}^{k}A))$ is $k+4$, and the largest power appearing in $\tr(\Phi(\sigma_{1}^{k}B))$ is $2$ if $k=0,1$ and $k$ if $k > 1$.  
\vspace{1pt}

\noindent\textbf{Example 2:} See Figure \ref{menasco}.

The braids $$(\sigma_{1})^{k}\sigma_{2}^{-1}\sigma_{3}\sigma_{2}\sigma_{3}^{-1}\sigma_{2}(\sigma_{1}^{-1})^{k}\sigma_{4}^{-1}\sigma_{2}^{-1}\sigma_{3}\sigma_{2}^{-1}\sigma_{3}^{-1}\sigma_{2}\sigma_{4}$$ and $$(\sigma_{1})^{k}\sigma_{2}^{-1}\sigma_{3}\sigma_{2}\sigma_{3}^{-1}\sigma_{2}(\sigma_{1}^{-1})^{k}\sigma_{4}\sigma_{2}^{-1}\sigma_{3}\sigma_{2}^{-1}\sigma_{3}^{-1}\sigma_{2}\sigma_{4}^{-1}$$ are related by a mutation and their $SKh$ differ for all $k \geq 1$.

We prove the result by showing that the traces of the Burau matrices of these two braids are different when we set $t = -1$.  First, it can be easily checked that the form for $(\sigma_{1})^{k}$ when $t = -1$ is:

$$\left(\begin{matrix} k+1 & -k & 0 & 0 & 0 \\ k & -k+1 & 0 & 0 & 0 \\ 0 & 0 & 1 & 0 & 0 \\ 0 & 0 & 0 & 1 & 0 \\ 0 & 0 & 0 &0&1\end{matrix}\right)$$
and the form for $(\sigma_{1}^{-1})^{k}$ when $t=-1$ is:

$$\left(\begin{matrix} 1-k & k & 0 & 0 & 0 \\ -k & k+1 & 0 & 0 & 0 \\ 0 & 0 & 1 & 0 & 0 \\ 0 & 0 & 0 & 1 & 0 \\ 0 & 0 & 0 &0&1\end{matrix}\right)$$

It can also be readily checked (for example, using Mathematica), that when $t = -1$,  the Burau matrices for $X= \sigma_{2}^{-1}\sigma_{3}\sigma_{2}\sigma_{3}^{-1}\sigma_{2}$, $Y = \sigma_{4}^{-1}\sigma_{2}^{-1}\sigma_{3}\sigma_{2}^{-1}\sigma_{3}^{-1}\sigma_{2}\sigma_{4}$, and $Z = \sigma_{4}\sigma_{2}^{-1}\sigma_{3}\sigma_{2}^{-1}\sigma_{3}^{-1}\sigma_{2}\sigma_{4}^{-1}$ are as follows:

$$ X = \left(\begin{matrix} 1 & 0 & 0 & 0 & 0 \\ 0 & 5 & -2 & -2 & 0 \\ 0 & 6 & -2 & -3 & 0 \\ 0 & 2 & -1 & 0 & 0 \\ 0 & 0 & 0 & 0 & 1 \end{matrix}\right)$$
$$ Y = \left(\begin{matrix} 1 & 0 & 0 & 0 & 0 \\ 0 & -3 & 2 & 4 & -2 \\ 0 & -6 & 4 & 6 & -3 \\ 0 & 0 & 0 & 1 & 0 \\ 0 & 2 & -1 & -2 & 2 \end{matrix}\right)$$
$$ Z = \left(\begin{matrix} 1 & 0 & 0 & 0 & 0 \\ 0 & -3 & 2 & 0 & 2 \\ 0 & -6 & 4 & 0 & 3 \\ 0 & -4 &2 & 1 & 2 \\ 0 & -2 & 1 & 0 & 2 \end{matrix}\right)$$

It can be calculated that $\tr(\Phi(\sigma_{1}^{k})X\Phi(\sigma_{1}^{-k})Y) = 6 + 8k + 16k^{2}$ and $\tr(\Phi(\sigma_{1}^{k})X\Phi(\sigma_{1}^{-k})Z) = 6 - 8k+16k^{2}$.  So for all $k \geq 1$, the traces of the Burau matrices are distinct.

\end{proof}

The negative flype examples mentioned in the note in the Introduction are also a direct calculation: their traces under the Burau representation differ as well.

\footnotesize
\bibliographystyle{amsalpha}
\bibliography{SKh}
\end{document}